\documentclass[reqno]{amsart}
\usepackage{amsrefs}

\usepackage{amsfonts}
\usepackage{amsmath}
\usepackage[arrow,matrix,tips, curve]{xy}
\usepackage{amssymb}
\usepackage{amsthm}
\usepackage{latexsym}
\usepackage{url}
\usepackage{verbatim}

\theoremstyle{plain}
\newtheorem{theorem}{Theorem}
\newtheorem{proposition}[theorem]{Proposition}
\newtheorem{lemma}[theorem]{Lemma}
\newtheorem{corollary}[theorem]{Corollary}
\newtheorem{conjecture}[theorem]{Conjecture}
\newtheorem{Hensel's Lemma}[theorem]{Hensel's Lemma}

\theoremstyle{definition}

\theoremstyle{remark}

\newcommand{\set}[1]{\left\{#1\right\}}

\newcommand{\Z}{\mathbb{Z}}
\newcommand{\Zp}{\mathbb{Z}_p}
\newcommand{\Qp}{\mathbb{Q}_p}
\newcommand{\oneunit}[1]{\left\langle#1\right\rangle}
\DeclareMathOperator{\ord}{ord}

\usepackage{etoolbox}
\makeatletter
\patchcmd{\@maketitle}
  {\ifx\@empty\@dedicatory}
  {\ifx\@empty\@date \else {\vskip3ex \centering\footnotesize\@date\par\vskip1ex}\fi
   \ifx\@empty\@dedicatory}
  {}{}
\patchcmd{\@adminfootnotes}
  {\ifx\@empty\@date\else \@footnotetext{\@setdate}\fi}
  {}{}{}
\makeatother

\begin{document}

\title[the discrete lambert map]{the discrete lambert map}

\author{Anne Waldo}

\date{July 29, 2014}

\address{Department of Mathematics and Statistics, 
Mount Holyoke College,
50 College Street, South Hadley,
  MA 01075, USA}

\email{waldo22a@mtholyoke.edu}

\author{Caiyun Zhu} 

\address{Department of Mathematics and Statistics, 
Mount Holyoke College,
50 College Street, South Hadley,
  MA 01075, USA}

\email{zhu22c@mtholyoke.edu}

\thanks{We would like to thank the Hutchcroft Fund of the Department of Mathematics and Statistics at Mount Holyoke for funding the summer research project in 2014. }


\begin{abstract}
The goal of this paper is to analyze the discrete Lambert map $x \to xg^x \pmod {p^e}$ which is important for security and verification of the ElGamal digital signature scheme. We use $p$-adic methods ($p$-adic interpolation and Hensel's Lemma) to count the number of solutions $x$ of $xg^x \equiv c \pmod {p^e}$ where $p$ is an odd prime and $c$ and $g$ are fixed integers. At the same time, we discover special patterns in the solutions.\end{abstract}

\maketitle
\section{Introduction}
A discrete logarithm is an integer $x$ solving the equation $g^x \equiv c \pmod p$ for some integers $c$, $g$, and for a prime $p$.  Finding discrete logarithms for large primes and fixed values for $c$ and $g$, referred to in this paper as the discrete logarithm problem (DLP), is thought to be difficult. The exponential function is used in different forms of public-key cryptography where the security depends on the difficulty of finding solutions to the DLP.  One particular class of cryptosystems where the DLP is important are digital signature schemes, which enable a message's recipient to verify the identity of the sender.

A specific digital signature scheme important for our paper is the ElGamal digital signature scheme, which is a public key system. For this system the values made public are $p$, $g$, $m$, and $h\equiv g^x \pmod p$, while the values known only to the sender are $y$ and $x$.The signature $(s_1,s_2)$ is computed as follows: $s_1 \equiv g^y \pmod p$ and $s_2 \equiv y^{-1}(m-xs_1) \pmod{p-1}$, where $m$ is the message, $p$ is a large prime, $g$ is a generator for $p$, $x\in\set{1,\dots,p-2}$, and $y\in\set{1,\dots,p-2}$ such that $\gcd(y,p-1)=1$.  The recipient of message $m$ also receives the signature $(s_1,s_2)$ and verifies the message by computing $v_1 \equiv h^{s_1}s_1^{s_2} \pmod p$ and $v_2 \equiv g^m \pmod p$. If $v_1 \equiv v_2 \pmod p$ then the signature is considered authentic.

In order to forge a signature, there are several methods with which to attack the system. One could solve the DLP by computing $x$ from $h \equiv g^x \pmod p$ for a fixed $g$, $h$ and prime $p$.  Another method is to fix $s_1$ and solve for $s_2$, requiring finding solutions to the congruence $s_1^{s_2} \equiv g^mh^{-s_1} \pmod p$, which is equivalent to solving another DLP since the right hand side of this congruence is a constant.  Both of these attacks are considered to be sufficiently hard and thus not feasible as a method of forgery.  A third method is to fix $s_2$ and solve for $s_1$, requiring finding solutions to the congruence $h^{s_1}s_1^{s_2} \equiv g^m \pmod p$.  Rewriting this congruence, we see that solving it for $s_1$ is equivalent to solving the congruence $s_1 (h^{{s_2}^{-1}})^{s_1} \equiv g^{M{s_2}^{-1}} \pmod {p}$ for $s_1$. Finally, setting
 $a=h^{{s_2}^{-1}}$ and $b=g^{M{s_2}^{-1}}$, we see that solving these congruences is equivalent to solving the congruence $s_1 a^{s_1} \equiv b \pmod p$ for $s_1$ with a fixed $a$ and $b$.  Due to its similarity to the Lambert W function \cite{corless} and to distinguish it from the DLP, we will refer to the map $s_1 \to s_1 a^{s_1} \pmod p$ as the discrete Lambert map.  Thus we define the discrete Lambert problem (DWP) to be the problem of finding integers $x$ such that $xg^x \equiv c \pmod p$ for fixed integers $g$ and $c$.

While the DLP has been studied extensively, the DWP has received very little attention although some introductory work has been done by Chen and Lotts on the DWP modulo $p$ \cite{Chen_Lotts}.  The lack of attention received by the DWP is in part because it is considered to be more difficult to solve than the DLP, but due to the implications that it has on the security of the ElGamal scheme we believe that it is important to study.

Finding exact formulas for the solutions seems extremely difficult, but counting the number of solutions for a fixed $g$ and $c$ and in an extended range of values for $x$ is much easier.  In addition, we can find patterns in the solutions that will give us insight into the DWP.  Beyond finding solutions and patterns modulo an odd prime $p$, we also wanted to look at solutions modulo $p^e$ in a similar fashion to what Holden and Robinson \cite{holden_robinson} do for the DLP.

\section{Counting Solutions}
We begin by looking at the DWP modulo $p$, counting the solutions and finding patterns.  The following theorems describe the number of solutions:
\begin{theorem}
	\label{caiyun_p}
If $p$ is an odd prime, $g$ a generator modulo $p$, and $c\not\equiv 0 \pmod{p}$, then for fixed $g$ and $c$, if we consider the function
	\begin{equation}
	f(x)=xg^x-c \equiv 0 \pmod{p}
	\end{equation}
	where $x\in\set{1,\dots,p(p-1) \mid x\not\equiv 0 \mod p}$, then the number of $x$ such that $f(x)\equiv 0 \pmod{p}$ is $p-1$ and the solution set forms a complete residue system modulo $p-1$.  In other words, the solutions are distinct modulo $p-1$.
\end{theorem}

\begin{proof} Since $g$ is a generator, we can take  the logorithm of equation (1) to get 
	\begin{equation}\log_gx+x \equiv \log_gc \pmod{p-1}.
	\end{equation}
	In order to show the solution set of equation (1) forms a complete residue system modulo $p-1$, we need to show there exist $p-1$  distinct solutions to equation (1), one for each $x_0$ such that
	\begin{equation}x \equiv x_0 \pmod{p-1}, \text{ for each } x_0\in \mathbb{Z}/(p-1) \mathbb{Z}.
	\end{equation}
	If we subtract equation (3) from equation (2), we get
	\begin{equation}\log_gx \equiv \log_gc-x_0 \pmod{p-1}.
	\end{equation}
	Then when we raise equation (4) to to power of $g$, we get 
	\begin{equation}x \equiv \frac{c}{g^{x_0}} \pmod{p}.
	\end{equation}
	Finally we can apply Chinese Remainder Theorem to equations (3) and (5), for each $x_0$, and conclude that there exist $p-1$ distinct solutions and they form a complete residue system modulo $p-1$.
	
\end{proof}

We can also look at what happens when $g$ is not a generator:
\begin{theorem}
	\label{anne_p}
	Let $p$ be an odd prime and $m=\ord_p(g)$. For fixed $g$ and $c$ such that $p\nmid g$ and $p\nmid c$, if we consider the function 
	$$f(x)=xg^x-c$$
	where $x\in\set{1,\dots,pm \mid x\not\equiv 0 \mod p}$, then the number of $x$ such that $f(x)\equiv 0 \pmod p$ is equal to $m$, and they are all distinct modulo $m$.
\end{theorem}

\begin{proof}
	Let 
		\begin{align} 
			x\equiv x_0 \pmod m \label{eq1}.
		\end{align} Then we have the following equivalent statements:
	\begin{align}
		f(x)=xg^x - c &\equiv 0 \pmod p\notag\\
		xg^{x_0} -c &\equiv 0 \pmod p\notag\\
		xg^{x_0} &\equiv c \pmod p\notag\\
		x &\equiv cg^{-x_0} \pmod p \label{eq2}.
	\end{align}  So for each $x_0 \in\set{1,\dots,m}$ there is an $x\in\set{1,\dots,p}$, and so by the Chinese Remainder Theorem on equations \eqref{eq1} and \eqref{eq2} there is exactly one $x\in\set{1,\dots,pm}$ such that $x$ is a zero of $f(x)$ where $x\equiv x_0 \pmod m$.  Hence, the number of zeros $f(x) \equiv 0 \pmod p$ is equal to $m$, and they are all distinct modulo $m$.
\end{proof}

\section{Interpolation} \label{interpolation}
In order to count solutions of the DWP modulo $p^e$, we need to interpolate the function $f(x) = xg^x - c$, defined on $x \in \Z$ to a function on $x \in \Zp$, for $p$ an odd prime and fixed $g, c \in \Zp$. However, interpolation is only possible when $g \in 1 + p\Zp$ \cite{holden_robinson}.  In order to apply the following theorem from Katok \cite{katok}, we need to show $f(x)=xg^x-c$ is uniformly continuous if $g\in1+p\Zp$. Then we can interpolate $f: \Z \rightarrow \Zp$ to a new uniformly continuous function $f_{x_0}: \Zp \rightarrow \Zp$.

\begin{theorem}[Thm. 4.15 of \cite{katok}]\label{katok}
Let E be a subset of $\Zp$ and let $\bar{E}$ be its closure. Let $f: E\rightarrow \Qp$ be a function uniformly continuous on E. Then there exists a unique function $F: \bar{E} \rightarrow \Qp$ uniformly continuous and bounded on $\bar{E}$ such that
$$ F(x)=f(x) \text{ if } x \in E.$$
\end{theorem}
\begin{proposition}
If $p$ is an odd prime, $c \in \Zp$ is fixed, and $g \in 1+p\Zp$, then $f(x)=xg^x-c$ is uniformly continuous for $x \in \Z.$
\end{proposition}
\begin{proof}
Suppose $g=1+pA$ where $A \in \Zp$.  We know that given any $\epsilon >0$, there exists an $N$ such that $p^{-N} < \epsilon$. Let $x,y \in \Z$ such that $$|x-y|_p \leq p^{-N} < p^{-(N-1)}=\delta,$$ or $(x-y) \in p^N\Zp$, and $x=y+bp^N$ where $b\in \Z$, then we need to show that
	$$|xg^x-c-(yg^y-c)|_p<\epsilon.$$
	
Note that
\begin{align*}
	g^{bp^N} &= (1+pA)^{bp^N}\\
	&= 1+ bp^NpA + \dots + (pA)^{bp^N}\\
	&\in 1+ p^N\Zp,
\end{align*}
so we know that $g^{bp^N} -1 \in p^N\Zp$, or $|g^{bp^N} -1|_p \leq p^{-N}$. Further, since $y\in\Z$, $|y|_p \leq 1$.  Also note that $|bp^Ng^{bp^N}|_p \leq p^{-N}.$
Now, consider \begin{align*}
		|xg^x-yg^y|_p & = |(y+bp^N)g^{y+bp^N}-yg^y|_p \\
			& =|yg^ybp^N+bp^Ng^{y+bp^N}-yg^y|_p \\
			& =|g^y|_p |yg^{bp^N}+bp^Ng^{bp^N}-y|_p, \text{ and since } g\in1+p\Zp ,\: |g^y|_p = 1\\
			&= |yg^{bp^N}+bp^Ng^{bp^N}-y|_p \\
			&= |(g^{bp^N} -1)y + bp^Ng^{bp^N}|_p\\
			&\leq \max\left(|g^{bp^N} -1|_p |y|_p, |bp^Ng^{bp^N}|_p \right)\\
			& \leq p^{-N}.
		\end{align*}
Hence, if $p$ is an odd prime, $c \in \Zp$ is fixed, and $g \in 1+p\Zp$, we have shown that $f(x)=xg^x-c$ is uniformly continuous for $x \in \Z$. 
\end{proof}

Now we can apply Theorem \ref{katok} of \cite{katok} to interpolate $f: \Z \rightarrow \Zp$ to a function $f_{x_0}: \Zp \rightarrow \Zp$.  If we let $\omega(g)$ be a $(p-1)^{th}$ of $1$ in $\Zp$ which is also called the Teichm\"uller character of $g$ and $\oneunit g\in1+p\Zp$, then we can rewrite $g=\omega(g)\oneunit{g}$ where $\oneunit{g}= \frac{g}{\omega(g)} \in 1+p\Z_p$. So we can consider a new function $f_{x_0}(x) = \omega(g)^{x_0}\oneunit{g}^x-c$, and we have the following proposition.



\begin{proposition}
	For an odd prime $p$, let $g\in \Zp$ such that $p\nmid g$ and $x_0 \in \Z/(p-1)\Z$, and let
	$$I_{x_0} = \set{x \in \Z \mid x \equiv x_0 \pmod{p-1}} \subseteq \Z.$$
Then 
	$$f_{x_0}(x) = x\omega(g)^{x_0}\oneunit{g}^x-c$$
	defines a uniformly continuous function on $\Zp$ such that $f_{x_0}(x) = f(x)$ whenever $x\in I_{x_0}$.

\end{proposition}

\section{Hensel's Lemma} \label{hensel}

\begin{lemma} [Hensel's Lemma, Cor.3.3 of~\cite{holden_robinson}] \label{hensel-lemma} 
Let $f(x)$ be a convergent power series in $\Zp[[x]]$ and let $a\in\Zp$ such that $\frac{df}{dx}(a)\not \equiv 0 \pmod{p}$ and $f(a) \equiv 0 \pmod p$.  Then there exists a unique $x \in \Zp$ for which $x \equiv a \pmod p$ and $f(x)=0$ in $\Zp$.
\end{lemma}

\begin{lemma}\label{lemma}
If we consider the function $$f_{x_0}(x)= x\omega(g)^{x_0}\exp(x\log (\oneunit{g}))-c$$ for any $a\in\Zp$ such that $f(a) \equiv 0 \pmod p$, then $\frac{df}{dx}(a)\not \equiv 0 \pmod{p}$.  
\end{lemma}
\begin{proof} 
Consider
	$$f(x)=xg^x-c \pmod{p}.$$
	If we take $x_0 \in \Z/m\Z$ where $m=\ord_p(g)$, we have
	$$f_{x_0}(x)= x\omega(g)^{x_0}\exp(x\log (\oneunit{g}))-c.$$ 
	Note that $\oneunit{g}\in 1+p\Z_p$. Furthermore, since $\log (\oneunit{g}) \in p\Zp$ then by the definition of the $p$-adic exponential function we know that $\exp(x\log(\oneunit{g})) \in 1+p\Zp$.  Taking the derivative of $f_{x_0}(x)$ (see proposition 4.4.4 of \cite{gouvea}), we have	
	\begin{align*}
		\frac{df_{x_0}}{dx}(x) & = \omega(g)^{x_0}\exp(x\log(\oneunit{g}))+x\omega(g)^{x_0}\exp(x\log(\oneunit{g}))\log(\oneunit{g}) \\
		& \equiv \omega(g)^{x_0}\exp(x\log(\oneunit{g})) \pmod{p}\\
		& \equiv  \omega(g)^{x_0} \pmod{p}\\
		&\not\equiv 0  \pmod{p}.
	\end{align*}
\end{proof}

%

\begin{proposition} \label{powerseries} For $p$ an odd prime, let $g \in \Zp^\times$ be fixed and let $m=\ord_p(g)$.  Then for every $x_0 \in \Z/m\Z$, there is exactly one solution to the function
$$f_{x_0}(x)=\omega(g)^{x_0} \oneunit{g}^x -c\equiv 0 \pmod p$$
for $x \in \Zp$.
\end{proposition}

\begin{proof} We know that $\oneunit{g} \equiv 1 \pmod{p}$, so the equation simplifies to
$$x\omega(g)^{x_0} \equiv c\pmod{p}.$$
For fixed $g$ and $x_0$, this has exactly one solution.

We know that $\oneunit{g}$ is in $1+p\Zp$, so we can say
\begin{eqnarray*}
\oneunit{g}^{x}=\exp(x \log(\oneunit{g}))=1&+&x\log(\oneunit{g})+
x^2\log( \oneunit{g})^2/2! \\
&+& \mbox{higher order terms in powers of }
\log(\oneunit{g}).
\end{eqnarray*}  By the definition of the $p$-adic logarithm we know $\log(\oneunit{g}) \in p\Zp$.  Since $$\lim_{i \to \infty} |\log(\oneunit{g})^i/i!|_p = 0,$$ we have a convergent power series.  We showed in Lemma \ref{lemma} that $f_{x_0}(x)$ satisfies the rest of the conditions of Hensel's Lemma, so we can apply the lemma to say there is a unique solution for $x\in \Zp$ such that $f_{x_0}(x)\equiv 0 \pmod p$.
\end{proof}

Now we can take Theorems \ref{caiyun_p} and \ref{anne_p} and generalize them to consider solutions modulo $p^e$.
\begin{theorem}[Generalization of Theorem \ref{anne_p}]
\label{anne_pe}
		Let $p$ be an odd prime and $m=\ord_p(g)$. For fixed $g$ and $c$ such that $p\nmid g$ and $p\nmid c$, if we consider the function 
	$$f(x)=xg^x-c$$
	where $x\in\set{1,\dots,p^em \mid x\not\equiv 0 \mod p}$, then the number of $x$ such that $f(x)\equiv 0 \pmod{p^e}$ is equal to $m$, and they are all distinct modulo $m$.
\end{theorem}

\begin{proof}
We can use Hensel's Lemma to count the number of solutions modulo $p^e$ given the number of solutions modulo $p$. In other words, the number of solutions to
$$f_{x_0}(x)=x\omega(g)^{x_0} \exp(x\log (\oneunit{g} ) ) -c \equiv 0\pmod{p^e}$$  is the same as the number of solutions to
$$f_{x_0}(x)= x\omega(g)^{x_0}\exp(x\log ( \oneunit{g})) -c\equiv 0 \pmod{p}$$ 
because of the bijection from the solution set of $f_{x_0}(x)\equiv 0$ modulo $p$ to the solution set modulo $p^e$.  We showed in Proposition \ref{powerseries} that there is exactly one $x_1\in\set{1,\dots,p}$ such that
$$x_1\omega(g)^{x_0} \oneunit{g}^{x_1}  \equiv c\pmod{p},$$ so using Hensel's Lemma there is exactly one $x_1\in\set{1,\dots,p^e}$ such that $$x_1\omega(g)^{x_0} \oneunit{g}^{x_1}  \equiv c\pmod{p^e}.$$   By the Chinese Remainder Theorem, there will be exactly one $x\in\set{1,\dots,p^em}$ such that 
$$x\equiv x_0 \pmod{m}$$ and
$$x\equiv x_1 \pmod{p^e}.$$
From the interpolation above we had $x\equiv x_0 \pmod m$,  and we know that for this $x\in\set{1,\dots,p^em}$:
$$f_{x_0}(x) = x\omega(g)^{x_0}\oneunit{g}^x -c\equiv 0 \pmod{p^e}.$$  Since there is exactly one such $x$ for each $x_0\in\set{1,\dots,m}$, there are $m$ solutions to $f(x) \equiv 0 \pmod{p^e}$.
\end{proof}

\begin{corollary}[Generalization of Theorem \ref{caiyun_p}]
	\label{caiyun_pe}
If $p$ is an odd prime, $g$ a generator modulo $p$, and $c\not\equiv 0 \pmod{p}$, then for fixed $g$ and $c$, if we consider the function
	\begin{equation}
	f(x)=xg^x-c
	\end{equation}
	where $x\in\set{1,\dots,p^e(p-1) \mid x\not\equiv 0 \mod p}$, then the number of $x$ such that $f(x)\equiv 0 \pmod{p^e}$ is $p-1$ and the solution set forms a complete residue system modulo $p-1$.
\end{corollary}
\begin{proof}
	Since $g$ is a generator modulo $p$, $m=\ord_p(g)=p-1$.  Then we can apply Theorem \ref{anne_pe} and there are $p-1$ solutions to $xg^x \equiv c \pmod{p^e}$ and they form a complete residue system modulo $p-1$  because they are distinct modulo $p-1$.
\end{proof}

\section{Patterns in the Solutions}
After counting the number of solutions to the DWP, we looked at patterns relating to $g$ and $c$ in the solutions modulo $p$ and modulo $p^e$.  One such pattern relates the solutions to the $c$ values associated with them:

\begin{theorem}
	Let $p$ be an odd prime and $m=\ord_p(g)$.  For fixed $g$ and $c$ such that $p\nmid g$ and $p\nmid c$, if we consider the function 
	$$f(x) = xg^x -c$$
	where $x\in\set{1,\dots,p^em \mid x\not\equiv 0 \mod p}$, then for any other $c'\in\set{1,\dots, p^{e-1}(p-1)}$, let $x_{i,c'}$ and $x_{j,c}$ for $1 \leq i,j \leq m$ index the $m$ solutions to
	$$x_{i,c'}g^{x_{i,c'}} \equiv c' \pmod{p^e}$$ and
	$$x_{j,c}g^{x_{j,c}} \equiv c \pmod{p^e}, \text{ respectively.}$$  
	If $c'\equiv x_{j,c} \pmod{p}$, then for each $x_{i,c'}$ there exists a unique $k$, $1\leq k \leq m$, and $x_{k,c}$ such that $x_{k,c} \equiv x_{i,c'} \pmod p$.
\end{theorem}

\begin{proof}	
	We know from Theorem \ref{anne_pe} that there are $m$ solutions to $f(x)\equiv 0 \pmod{p^e}$.  We will show that for fixed $i,j$ that if $c'\equiv x_{j,c} \pmod{p}$, then for all $x_{i,c'}$ there exists a unique $x_{k,c}$ such that $x_{i,c'} \equiv x_{k,c} \pmod p$.  To begin, we have the equations
	\begin{align}
		x_{i,c'}g^{x_{i,c'}} \equiv c' \pmod{p^e} \label{c1}
	\end{align} and
	\begin{align}
		x_{j,c}g^{x_{j,c}} &\equiv c \pmod{p^e} \text{, or equivalently}\notag\\
		x_{j,c} &\equiv cg^{-x_{j,c}} \pmod{p^e} \label{c2}.
	\end{align}  Since $x_{k,c}$ ranges through the solutions to
	\begin{align}
		x_{k,c}g^{x_{k,c}} \equiv c \pmod{p^e} \label{c4}
	\end{align} where $k\in\set{1,\dots,m}$ and by Theorem \ref{anne_pe} the solutions $x_{k,c}$ are all distinct modulo $m$, we can choose $x_{k,c}$ specifically by the Chinese Remainder Theorem so that 
	\begin{align}
		x_{i,c'} &\equiv x_{k,c} -x_{j,c} \pmod m. \label{c3}
	\end{align}
	This use of the Chinese Remainder Theorem will give a unique $x_{k,c}$ for each $x_{i,c'}$ because $x_{i,c'}$ and $x_{j,c}$ are both fixed.
	 Now, we originally said that $c' \equiv x_{j,c} \pmod p$, so we have the following equivalent statements from equations \eqref{c1} and \eqref{c2}:
	\begin{align*}
		x_{i,c'}g^{x_{i,c'}} &\equiv cg^{-x_{j,c}} \pmod p.
	\end{align*}  We can substitute $c$ with equation \eqref{c4}:
	\begin{align*}
		x_{i,c'}g^{x_{i,c'}} &\equiv (x_{k,c}g^{x_{k,c}})g^{-x_{j,c}} \pmod p\\
		x_{i,c'}g^{x_{i,c'}} &\equiv x_{k,c}g^{x_{k,c}-x_{j,c}} \pmod p.
	\end{align*} Finally, using equation \eqref{c3} we simplify to:
	\begin{align*}
		x_{i,c'} &\equiv x_{k,c} \pmod p.
	\end{align*} Thus, for all $i\in\set{1,\dots,m}$, there is some unique $k$ such that $x_{i,c'} \equiv x_{k,c} \pmod p$ when $c'\equiv x_{j,c} \pmod p$.
\end{proof}

Another pattern we found involves the sum of the solutions modulo $p$ and modulo $m$:

\begin{theorem}
		Let $p$ be an odd prime and $m=\ord_p(g)$.  For fixed $g$ and $c$ such that $p\nmid g$ and $p\nmid c$, if we consider the function 
		$$f(x) = xg^x -c$$
		where $x\in\set{1,\dots,p^em \mid x\not\equiv 0 \mod p}$, then for each $c$ there are $m$ solutions, $x_1,\dots,x_m$, to $f(x)\equiv 0 \pmod{p^e}$ such that 
		$$\sum_{i=1}^{m} x_i \equiv 0 \pmod{p},$$ and for odd $m$
		$$\sum_{i=1}^{m} x_i \equiv 0 \pmod m.$$
\end{theorem}

\begin{proof}
	We know from Theorem \ref{anne_pe} that there are $m$ solutions to $f(x)\equiv 0 \pmod{p^e}$. First, we will show that for each $c$ the solutions sum as follows: $$\sum_{i=1}^m x_i \equiv 0 \pmod{p}.$$  Since we said in Theorem \ref{anne_pe} that for each $i\in\set{1,\dots,m}$, $x_i \equiv x_0 \pmod m$ where $x_0\in\set{1,\dots,m}$, we can let $x_i\equiv i \pmod m$.  Then we know $x_i \equiv cg^{-i} \pmod{p}$.  Taking the sum of these $x_i$ gives us:
	\begin{align*}
		\sum_{i=1}^{m} x_i &\equiv \sum_{i=1}^{m} cg^{-i} \pmod{p}\\
		&\equiv \sum_{i=0}^{m-1} cg^{-i} \pmod p\\
		&\equiv c\left(\frac{1-g^{-m}}{1-g}\right) \pmod{p}\\
		&\equiv c\left(\frac{1-1}{1-g}\right) \pmod{p}\\
		&\equiv 0 \pmod{p}.
	\end{align*} 
	Thus, $\sum_{i=1}^m x_i \equiv 0 \pmod{p}$ for each $c$.
	
	Now, we will show that $\sum_{i=1}^{m} x_i \equiv 0 \pmod m$ when $m$ is odd.  Again, we have that $x_i\equiv i \pmod m$.  For each $i\in\set{1,\dots,m}$, we have
	\begin{align*}
		\sum_{i=1}^{m} x_i &\equiv \sum_{i=1}^{m} i \pmod m\\
		&\equiv \frac{m(m+1)}{2} \pmod m\\
		&\equiv 0 \pmod m.
	\end{align*}  Thus, $\sum_{i=1}^m x_i \equiv 0 \pmod{m}$.
\end{proof}

We conjecture that the same pattern of sums holds for solutions modulo $p^e$ and modulo $\ord_{p^e}(g)$, based on the evidence for all odd primes $p \leq 17$ and $1\leq e \leq 4$.

\begin{conjecture}
		Let $p$ be an odd prime, $m_p=\ord_p(g)$ and $m_{p^e}=\ord_{p^e}(g)$.  For fixed $g$ and $c$ such that $p\nmid g$ and $p\nmid c$, if we consider the function 
		$$f(x) = xg^x -c$$
		where $x\in\set{1,\dots,p^em_p \mid x\not\equiv 0 \mod p}$, then for each $c$ there are $m_p$ solutions, $x_1,\dots,x_{m_p}$, to $f(x)\equiv 0 \pmod{p^e}$ such that 
		$$\sum_{i=1}^{m_p} x_i \equiv 0 \pmod{p^e}$$ and for odd $m$
		$$\sum_{i=1}^{m_p} x_i \equiv 0 \pmod{m_{p^e}}.$$
\end{conjecture}

We also looked some patterns for fixed $x$ and variable $c$.

\begin{theorem}
	Let $p$ be an odd prime.  For a fixed $x\in\set{1,\dots,p^e}$ and for $p\nmid g$ and $c\in\set{1,\dots, p^{e-1}(p-1)}$, if we consider $xg^x \equiv c \pmod{p^e}$ and let $x(g^{-1})^x \equiv c' \pmod{p^e}$, then $c\cdot c' \equiv x^2 \pmod{p^e}$.  Furthermore, if we let $x(-g)^x \equiv c'' \pmod{p^e}$ then $c'' \equiv (-1)^xc \pmod{p^e}$.
\end{theorem}

\begin{proof}
	First, we will show that $c\cdot c' \equiv x^2 \pmod{p^e}$.  Since $c\equiv xg^x \pmod{p^e}$ and $c'\equiv x(g^{-1})^x \pmod{p^e}$, we can say that
	\begin{align*}
		c\cdot c' &\equiv (xg^x)(x(g^{-1})^x) \pmod{p^e}\\
		&\equiv x^2(g^x)(g^{-x}) \pmod{p^e}\\
		&\equiv x^2 \pmod{p^e}.
	\end{align*}Hence, $c\cdot c' \equiv x^2 \pmod{p^e}$.  Now, we need to show that $c''\equiv(-1)^xc\pmod{p^e}$. We have
	\begin{align*}
		c'' &\equiv x(-g)^x \pmod{p^e}\\
		&\equiv x(-1)^xg^x \pmod{p^e}\\
		&\equiv (-1)^xxg^x \pmod{p^e}\\
		&\equiv (-1)^xc \pmod{p^e}.
	\end{align*} Thus $c'' \equiv (-1)^xc\pmod{p^e}$.
\end{proof}

\begin{proposition}
Let $p$ be an odd prime and $g$ be a generator modulo $p^e$. If $c= \frac{p^e+p^{e-1}}{2}$,  then  $x= \frac{p^e-p^{e-1}}{2}$ is one of the solutions to 
	 $$xg^x \equiv c \pmod{p^e}. $$
\end{proposition}

 \begin{proof}
By hypothesis, we see that
	\begin{align*}
		xg^x- c & = \frac{p^e-p^{e-1}}{2}g^\frac{p^e-p^{e-1}}{2} - \frac{p^e+p^{e-1}}{2}\\
		&\equiv \frac{p^e-p^{e-1}}{2}(p^e-1)-\frac{p^e+p^{e-1}}{2} \pmod{p^e} \\
		&= \frac{p^{2e}-p^{2e-1}-p^e+p^{e-1}-p^e-p^{e-1}}{2} \pmod{p^e}\\
		&=\frac{p^{2e}-p^{2e-1}-2p^e}{2} \pmod{p^e}\\
		&= \frac{p^e(p^e-p^{e-1}-2)}{2} \pmod{p^e}\\
		&= \frac{p^e(p^{e-1}(p-1)-2))}{2} \pmod{p^e}\\
		&\equiv 0 \pmod{p^e}.
	\end{align*}	
Note that if $g$ is an generator modulo $p^e$,  $\ord_{p^e}(g)=p^e-p^{e-1}$, thus $g^{\frac{p^e-p^{e-1}}{2}}\equiv p^e-1 \pmod {p^e}$ because $(p^e-1)^2\equiv 1\pmod{p^e}$.
\end{proof}

\begin{proposition}
Let $n\geq 2$ and  $n \in \Z^+ $. If $\gcd (p, n) = 1$ and $p$ is an odd prime, then 
$$  
  	\ord_{p^e}(p-1)^n =\left\{
     \begin{array}{ll}
	     p^{e-1} & \quad n \text{ is even} \\
 	      2p^{e-1} & \quad n \text{ is odd.}
     \end{array}
    \right.
     $$
\end{proposition}

\begin{proof}
We will prove this by inducting on $e$. \\[.1in]For our base case, let $e=1$:\\[.1in]
When $n$ is even:
	\begin{align*}
		(p-1)^{n} &=1-np +  \frac{np(np-1)}{2}p^2 +\dots+ p^{np}\\
		&= 1-mp \\ 
		&\equiv 1 \pmod{p},
	\end{align*}  

where $m \in \mathbb{Z}$.\\[.15in]
When $n$ is odd: 
	\begin{align*}
		(p-1)^{2n} &=1+2np+\frac{2np(2np-1)}{2}p^2 +\dots+p^{2np}\\
		&=1+ap \\
		&\equiv 1 \pmod{p}, \text{ and} 
	\end{align*}
	\begin{align*}
		(p-1)^{n} &=-1+np - \frac{np(np-1)}{2}p^2 + \dots+ p^{np}\\
		&= -1 + bp \\
		&\equiv p-1 \pmod{p} \\ 
		&\not\equiv 1 \pmod{p},
	\end{align*}
	
where $a,b \in \mathbb{Z}$.\\[.1in]
So our base case holds:
	$$   
  	\ord_{p}(p-1)^n =\left\{
     \begin{array}{ll}
  	   1 & \quad n \text{ is even} \\
  	     2 & \quad n \text{ is odd.}
     \end{array}
    \right.
     $$
For our inductive hypothesis, we assume the following:
	$$   
  	\ord_{p^e}(p-1)^n =\left\{
     \begin{array}{ll}
   	  p^{e-1} & \quad n \text{ is even} \\
   	    2p^{e-1} & \quad n \text{ is odd.}
     \end{array}
    \right.
     $$
Now, in our inductive step we need to show:
	$$
	\ord_{p^{e+1}}(p-1)^n =\left\{
	\begin{array}{ll}
		p^e & \quad n \text{ is even} \\
		2p^e & \quad n \text{ is odd.}
	\end{array}
	\right.
	$$
	
When $n$ is even:
	\begin{align*}
		(p-1)^{np^e} &=1-np^ep +  \frac{np^e(np^e-1)}{2}p^2 +\dots+ p^{np^e}\\
		&= 1-kp^{e+1} \\ 
		&\equiv 1 \pmod{p^{e+1}},
	\end{align*}  

where $k \in \mathbb{Z}$.\\[.1in]
When $n$ is even, let $x$ be the least integer such that the following equivalent equations hold:
	\begin{eqnarray*}
		(p-1)^{xn} &\equiv 1\pmod{p^{e+1}}.\\
		1-xnp +  \frac{xn(xn-1)}{2}p^2 + \dots + p^{xn} &\equiv 1 \pmod{p^{e+1}}.\\
		-xnp +  \frac{xn(xn-1)}{2}p^2 + \dots + p^{xn} &\equiv 0 \pmod{p^{e+1}}.\\
		    							 px(-n+dp) & \equiv 0   \pmod{p^{e+1}}.
	\end{eqnarray*}
where $d \in \mathbb{Z}$. Since $\gcd(p,n)=1$, then $p \nmid -n+dp$. Therefore $p^{e} \mid x$, hence the least $x=p^{e}=\ord _{p^{e+1}}(p-1)^n$. The proof for showing $\ord_{p^e+1}(p-1)^n =2p^e$ when $n$ is odd is a parallel to the case when $n$ is even.\\[.1in]
Therefore $p^e$ and $2p^e$ are the least integers such that
	$$
	\left\{
	\begin{array}{ll}
		(p-1)^{np^e} \equiv 1 \pmod{p^{e+1}}  & \quad n \text{ is even} \\
		(p-1)^{2np^e} \equiv 1 \pmod{p^{e+1}} & \quad n \text{ is odd}.
	\end{array}
	\right.
	$$

\end{proof}

\section{Conclusions and Future Work}

Following Holden and Robinson \cite{holden_robinson}, we counted solutions to the discrete Lambert problem modulo powers of a prime $p$ and we found very similar results regarding the number of solutions for $x$ in $\set{1,\dots,p^e(p-1)}$ and  $\set{1,...,p^em}$ where $m$ is the multiplicative order of $g$ modulo $p$. For a given $g$ the value $m$ is very important in understanding the number of solutions to the DWP. In addition, we found how solutions modulo $p$ relate to $c$, as well as some special properties between the sum of the solutions and $p^e$. We also found that when $g$ is a generator modulo $p^e$ there is a special $(x,c)$ that satisfies the DWP.


According to Chen and Lotts, when $g=(p-1)$, the solutions to the DWP modulo $p$ are very predictable (see Section 3.4 \cite {Chen_Lotts}). Therefore it is not an good choice to use in a cryptosystem. However, they did not consider the solutions to the DWP modulo $p^e$. Due to the change in the multiplicative order of $p-1$ modulo $p^e$, the patterns in the solutions to the DWP become erratic and cannot be foreseen as far as we can tell.

We should mention that since this work was completed Dara Zirlin \cite{zirlin} has extended our research to the case where $p=2$ and has also counted the number of fixed points and two-cycles of the discrete Lambert map for all primes $p$. In particular, she has counted the number of solutions $x$ to $xg^x \equiv x \pmod {p^e}$ and the number of solutions $(h,a)$ to the system of congruences:
\begin{center}
	$hg^h \equiv a \pmod{p^e}$  and $ag^a \equiv h \pmod{p^e}$\\
\end{center}
where $x$, $a$ and $h$ range through the appropriate sets of integers, $g$ is fixed and $p$ is any prime.
\section{Acknowledgements}
We would like to thank Professor Joshua Holden and Professor Margaret Robinson for their guidance and support throughout our project during the summer of 2014.



\begin{bibdiv}
\begin{biblist}
\bib{Chen_Lotts}{article}{
	title = {Structure and Randomness of the Discrete Lambert Map}, 
	volume = {13},

	number = {1},
	journal = {Rose-Hulman Undergraduate Mathematics Journal},
	author = {J. Chen},
        author = {M. Lotts},
	year = {Spring 2012},
	pages = {64--99}
}
\bib{corless}{article}{
	title = {On the Lambert W function},
	journal = {Advances in Computational Mathematics},
	volume = {5},
	author={R. M. Corless},
	author={G. H. Gonnet},
	author = {D. E. G. Hare},
	author = {D. J. Jeffrey},
	author = {D. E. Knuth},
	year = {1996},
	pages = {329-359}	
}

\bib{holden_robinson}{article}{
	title = {Counting Fixed Points, Two-Cycles, And Collision of the Discrete Exponential Function Using p-adic Methods}, 
	
	
	journal = {Journal of the Australian Mathematical Society},
	author = {J. Holden},
        author = {M. Robinson},
	year = {2010},
	
}
\bib{gouvea}{book}{
	edition = {2},
	title = {p-adic Numbers: An Introduction},
	isbn = {3540629114} 
	publisher = {Springer},
	author = {F. Q. Gouvea},
	year = {1997-07}
}
\bib{katok}{book}{
	edition = {1},
	title = {p-adic Analysis Compared with Real},
	series={Student Mathematical Library},
	isbn = {9780821842201},
	publisher = {American Mathematical Society},
	author = {S. Katok},
	year = {2007}
}
\bib{koblitz}{book}{
	edition = {2},
	title = {$p$-adic Numbers, $p$-adic Analysis, and
          {Zeta-Functions}},
        series = {Graduate Texts in Mathematics},
	isbn = {0387960171},
	publisher = {Springer},
	author = {N. Koblitz},
	year = {1984-07}
}
\bib{zirlin}{article}{
	title = {Problems motivated by Cryptology:
 Counting fixed points and two-cycles of the discrete Lambert Map}, 
	
	
	journal = {Undergraduate thesis presented to the Mathematics and Statistics Department at Mount Holyoke College },
	author = {D. Zirlin},
       
	year = {2015},
	
}

\end{biblist}
\end{bibdiv}
\end{document}